\documentclass[12pt,english]{article}
\usepackage[letterpaper]{geometry}
\usepackage{amsmath,amssymb,amsthm}
\usepackage{setspace}
\usepackage{babel}
\usepackage{float}

\theoremstyle{plain}
\newtheorem{theorem}{Theorem}
\newtheorem*{theorem*}{Theorem}
\newtheorem{lemma}{Lemma}
\newtheorem{corollary}{Corollary}
\newtheorem*{corollary*}{Corollary}
\newtheorem{definition}{Definition}
\newtheorem*{remark*}{Remark}

\def\mathscr{\mathfrak}

\newcommand{\ds}{\displaystyle}

\def\R{{\mathbb R}}

\def\C{{\mathbb C}}

\def\Ker{\mathrm{Ker\,}}
\def\Ran{\mathrm{Ran\,}}

\def\llangle{\langle\!\langle}
\def\rrangle{\rangle\!\rangle}

\def\e{\varepsilon}

\def\s{\sigma}

\def\T{\mathcal{T}}
\def\L{\mathcal{L}}

\def\A{\mathcal{A}}
\def\B{\mathcal{B}}
\def\X{\mathcal{X}}

\def\F{\mathcal{F}}
\def\pfi{\varphi}
\def\D{\mathcal{D}}

\def\-{\backslash}
\def\ohat{\widehat{\otimes}}

\def\ds{\displaystyle}

\makeatletter

\date{}

\makeatother

\begin{document}

\title{Positive semigroups and algebraic Riccati equations\\ in Banach spaces}

\author{Sergiy Koshkin\\
 Computer and Mathematical Sciences\\
 University of Houston-Downtown\\
 One Main Street, \#S705\\
 Houston, TX 77002\\
 e-mail: koshkins@uhd.edu}

\maketitle
\begin{abstract}\

We generalize Wonham's theorem on solvability of algebraic operator Riccati equations to Banach spaces, namely there is a unique stabilizing solution to $A^*P+PA-PBB^*P+C^*C=0$ when $(A,B)$ is exponentially stabilizable and $(C,A)$ is exponentially detectable. The proof is based on a new approach that treats the linear part of the equation as the generator of a positive semigroup on the space of symmetric operators from a Banach space to its dual, and the quadratic part as an order concave map. A direct analog of global Newton's iteration for concave functions is then used to approximate the solution, the approximations converge in the strong operator topology, and the convergence is monotone. The linearized equations are the well-known Lyapunov equations of the form $A^*P+PA=-Q$, and semigroup stability criterion in terms of them is also generalized.
\bigskip


\end{abstract}

\section{Introduction}\label{S1}

In this paper we generalize Wonham's classical theory of positive solvability for algebraic matrix Riccati equations to Banach spaces. It provides a constructive solution to the linear quadratic optimal control problem in Banach spaces that reduces to solving a sequence of linear operator equations \cite[12.3]{Wh}. A generalization to Hilbert spaces was given by Zabczyk in 1970s \cite{Zb0}, but for Banach spaces no general theory appears to exist despite the widespread use of the operator Riccati equations in the optimal control theory. For a Banach space $X$ the operators appearing in the Riccati equation are not symmetric positive definite operators on $X$, they map from $X$ to $X^*$, but there are suitable notions of symmetricity and positive definiteness for them. The iterative process used in our solution goes back to Kleinman, who used it for matrix Riccati equations \cite{Kl}, Wonham in \cite[12.3]{Wh} gave more general conditions under which it works. It consists of solving a sequence of Lyapunov equations of the form $A^*P+PA=-Q$, where $A$ is the generator of a $C_0$ semigroup on a Banach space $X$, and $Q: X\to X^*$ is symmetric and positive definite. 

Two major issues complicate the solution theory on Banach spaces. One is the absence of positive definite isomorphisms, like the identity operator on a Hilbert space, which allow non-degenerate approximations. These can be sidestepped however. The second issue is deeper. The Riccati equations involve operators on spaces of operators that lack regularity properties in the the strong and the weak operator topologies. It turns out that the suitable topology is the weak*, and explicitly or implicitly one has to work with Banach preduals to spaces of operators. For Hilbert spaces the predual can be identified with the space of trace class operators, but in the Banach case it is described as a projective tensor product. Tensor products of Banach spaces are rarely used and relatively little known in control theory applications, which helps explain why the operator Riccati equations are either considered in Hilbert spaces only \cite{Ben,Zb}, or in some special cases \cite{GvN,PvN,RM}.

We give a novel interpretation of infinite dimensional Riccati equations in terms of Lyapunov semigroups and their generators, see \cite{Damm} for a similar approach with matrices. As pointed out in \cite{K} and  \cite{K2} the theory of Lyapunov equations implicitly relies on the fact that the left hand side is the generator of a semigroup $\T(t)P:=T^*(t)PT(t)$ on the space of symmetric bounded operators $\B_s(X,X^*)$. We called it the Lyapunov semigroup, and one of its attractive properties is that it is positive, i.e. it preserves the positive definiteness of $P$. Unfortunately, this semigroup is not $C_0$ if $A$ is unbounded, even when $X$ is a Hilbert space, which obstructs its use for analytic reasons. However, in \cite{K2} we proved that it is always adjoint to a $C_0$ semigroup on a predual to $\B_s(X,X^*)$, i.e. $\T(t)$ is always a $C_0^*$ semigroup \cite[3.1]{BR}. This property allowed us to overcome the analytic obstructions, but the price for using it is the appearence of tensor products as preduals. We should note that the original motivation for studying $C_0^*$ semigroups came from automorphism flows on von Neumann algebras, a particular case of Lyapunov semigroups when $T(t)$ is unitary.

In this paper we extend an explicit description of the Lyapunov semigroup generators \cite{Fr} to non-reflexive Banach spaces, see Theorem \ref{Lyappi}, and use it to solve operator Riccati equations. In addition, we also generalize a stability criterion in terms of Lyapunov equations \cite{K2} to non-reflexive spaces, see Theorem \ref{Wonh}, the matrix case is due to Wonham \cite[12.4]{Wh}.

That Riccati equation can be written as $\A P+\Phi(P)=0$, where $\A$ generates a positive $C_0^*$ semigroup, is only half of the solution however. The other half comes from the fact that the quadratic part $\Phi(P)$ is Gateaux differentiable and concave in the sense of partial order on the space of operators, a generalization of matrix concavity \cite{Ando}. It explains why Kleinman's method led to monotone convergence of approximations, as shown in \cite{Damm} it amounts to an analog of Newton's iteration on the space of matrices.

Applications to control theory single out a particular class of solutions to the equation that translate into asymptotically stabilizing controls. It turns out that such stabilizing solutions $P$ can also be characterized abstractly, for them the formal derivative $\A+\Phi'(P)$ generates a stable semigroup. We found it more illuminating to be general and prove existence of stabilizing solutions for a class of general quasi-linear concave equations, Theorem \ref{NewtRicc}. Aside from admitting a more streamlined argument, the result is of interest in its own right because it may apply to other situations, e.g. to reaction-diffusion equations with concave non-linear parts \cite[4.2]{Leu}. After everything is said and done our final result can be stated (almost) identically to Wonham's \cite[12.3]{Wh}: if $A-BK$ and $A-LC$ generate exponentially stable semigroups for some $K$ and $L$, the algebraic Riccati equation $A^*P+PA-PBB^*P+C^*C=0$ has a unique positive definite stabilizing solution $P$, see Theorem \ref{WonhBan}. 

The paper is organized as follows. Section \ref{S4} serves as preliminaries, we introduce the main concepts of the paper and state our main results. In particular, the linear quadratic optimal control problem in a Banach space $X$ is recalled, and how it leads to the algebraic operator Riccati equation. Analytic difficulties related to its interpretation and solution are also discussed. Sections \ref{S5} and \ref{S3} lay the analytic groundwork for the proof which is completed in Section
\ref{S6}. In particular, Section \ref{S5} introduces projective tensor products, which are preduals to the spaces of bounded operators between Banach spaces, considers topologies in which the Lyapunov semigroups are continuous, characterizes their generators, investigates properties of positive definite symmetric operators, and continuity of quadratic maps on monotone sequences of operators. In Section \ref{S3} we prove an abstract result on existence of stabilizing solutions to quasilinear concave equations in ordered spaces, and in Section \ref{S6} we deduce  from it our generalization of the Wonham's theorem.

\section{Optimal control and operator Riccati equations}\label{S4}

This section provides motivation for the theory developed in the paper and states our main results. We start by describing the optimal control problem that leads to algebraic operator Riccati equations, then introduce the main concepts needed to interpret and solve them, and conclude with discussion and examples. 

The linear quadratic regulator problem with infinite horizon can be described as follows. Let $X$ be a real Banach space, called the state space, and $A$ be the generator of a $C_0$ semigroup on $X$ with domain $\D_A$. Let $U$ and $V$ be real Hilbert spaces, called the control (input) and the observation (output) space respectively, and let $B:U\to X$ and $C:X\to V$ be bounded linear operators. Consider a linear system written formally as
\begin{equation}\label{LinSys}
\begin{cases}\dot{x}=Ax+Bu,\,x(0)=x_0\\y=Cx.\end{cases}
\end{equation}
Here $u:[0,\infty)\to U$ is a control input meant to steer the state $x(t)$ so as to fulfill some objective. The state itself is not available for direct observation, only some reduction of it $y(t)$ is. One typical control objective is optimal asymptotic stabilization, that is driving the state towards the origin asymptotically while minimizing a "cost". A popular way to do so is to minimize a quadratic cost functional 
\begin{equation}\label{QuadCost}
J[u,x_0]:=\int_0^\infty(y,y)_V+(u,u)_U\,dt\to\min,
\end{equation}
where $(\,,\,)_U$, $(\,,\,)_V$ are the inner products on $U$ and $V$ respectively. By analogy to the Hilbert case \cite[12.1]{Wh}, \cite[IV.4.1]{Zb}, one can show formally that the optimal control can be found in the feedback form $u(t)=-B^*Px(t)$, where $P$ is a bounded symmetric positive definite operator that solves the algebraic operator Riccati equation:
\begin{equation}\label{OpRicc}
A^*P+PA-PBB^*P+C^*C=0.
\end{equation}
Note that when the feedback control is applied, again formally, the state evolution equation becomes $\dot{x}=(A-BB^*P)x$, so the control objective can be accomplished only if $A-BB^*P$ generates a stable semigroup. When $X$ is a Banach space $P$ has to be interpreted as an operator from $X$ to $X^*$, but the usual Hilbert notions generalize naturally to such operators.
\begin{definition}\label{OrdConc} Let $\B(X,Y)$ denote the Banach space of bounded linear operators from a Banach space $X$ to a Banach space $Y$ with the induced norm. Denoting $\langle\,,\,\rangle$ the duality pairing between $X$ and $X^*$ we call 
$P\in\B(X,X^*)$ {\sl symmetric} if $\langle Px,y\rangle=\langle Py,x\rangle$ for all $x,y\in X$, or equivalently $P^*\big|_{X}=P$ (where $X$ is identified with a subspace of $X^{**}$). We call $P$ {\sl positive definite} if $\langle Px,x\rangle\geq0$ for all $x\in X$. The subspace of symmetric operators is denoted $\B_s(X,X^*)$, and the cone of positive definite ones in it $\B_s^+(X,X^*)$. Notations $\B_s(X^*,X)$ and $\B_s^+(X^*,X)$ are defined similarly.
\end{definition}
Most authors dealing with operator Riccati equations assume $X$ to be Hilbert (see however \cite{vN2,PvN,RM}), in which case $X^*$ can be identified with $X$, and $P$ is a self-adjoint operator on $X$. However, in applications the state space is dictated by the problem, and is often not a Hilbert space. It may seem natural to also allow the control and the observation spaces $U$ and $V$ to be non-Hilbert, but that leads to no real generalization as long as the cost functional remains quadratic. If one simply replaces the inner products in \eqref{QuadCost} by positive definite quadratic forms we can use them to define inner products, which brings us back to the Hilbert spaces $U$ and $V$.

As usual, we identify $U$ and $V$ with their duals, and set $N:=BB^*\in\B_s^+(X^*,X)$ and $Q:=C^*C\in\B_s^+(X,X^*)$ for brevity. Since $A$ and $A^*$ are only densely defined some care is needed to define solutions to \eqref{OpRicc}.  The traditional interpretation is to require that for all $x,y\in\D_A$ one has 
\begin{equation}\label{DualParRicc}
\langle Px,Ay\rangle+\langle PAx,y\rangle-\langle Px,NPy\rangle+\langle Qx,y\rangle=0.
\end{equation}
In the Hilbert case this is called inner product Riccati equation \cite[IV.4.2]{Zb}, so it is natural in general to call it the duality pairing Riccati equation. It is considered e.g. in \cite{vN2,PvN,RM}. We will use a different interpretation of 
\eqref{OpRicc} which treats $P\mapsto A^*P+PA$ as the generator of a positive semigroup.
\begin{definition} 
Let $T(t)$ be a $C_0$ semigroup on $X$ with the generator $A$. The semigroup 
$\T(t)P:=T^*(t)\,P\,T(t)$ on $\B(X,X^*)$ and its restriction to $\B_s(X,X^*)$ will be called the {\sl Lyapunov semigroup} of $T(t)$, and its generator will be denoted $\L_A$ and called the {\sl Lyapunov generator}.
\end{definition}
\noindent The semigroup property of $\T(t)$ is straightforward to verify as is its positivity, indeed 
$\langle\bigl(\T(t)P\bigr)x,x\rangle=\langle PT(t)x,T(t)x\rangle\geq0$ for positive definite $P$. But it is not obvious in general that it is continuous in some useful sense, for instance it is never $C_0$ for unbounded $A$. However, it was shown in \cite{K2} that Lyapunov semigroups are adjoint to $C_0$ semigroups on preduals to $\B(X,X^*)$ and $\B_s(X,X^*)$, in other words they are always $C_0^*$ semigroups \cite[3.1]{BR}. 

This allows us to make sense of the generator $\L_A$ and use its standard properties. Moreover, it turns out, see Theorem \ref{ComGen}, that the domain of $\L_A$ consists of operators $P$ that satisfy $P(\D_A)\subseteq\D_{A^*}$ and have $\|A^*P+PA\|<\infty$. Then $\L_AP$ is the extension of the bounded operator $A^*P+PA$ from $\D_A$ to the entire space. Therefore, we can interpret equation \eqref{OpRicc} as $\A P+\Phi(P)=0$ with $\A=\L_A$ and $\Phi(P)=-PNP+Q$, and view solutions as elements of $\D_\A$ that satisfy it literally. It also follows from Theorem \ref{ComGen} that solutions so defined are the same as solutions to the duality pairing Riccati equation. Our conditions for their existence are in terms of stabilizability and detectability, which for matrices go back to Wonham \cite[12.6]{Wh}.

\begin{definition} Let $A$ be a generator of a $C_0$ semigroup on $X$ and $B\in\B(U,X)$, $C\in\B(X,V)$. The pair $(A,B)$ is called {\sl exponentially stabilizable} if there exists $K\in\B(X,U)$ such that $A-BK$ generates an exponentially stable $C_0$ semigroup, and the pair $(C,A)$ is called {\sl exponentially detectable} if there exists $L\in\B(V,X)$ such that $A-LC$ generates an exponentially stable $C_0$ semigroup.
\end{definition}
\noindent Note that disregarding continuity stabilizability is formally dual to detectability, i.e. $(A,B)$ is exponentially stabilizable whenever $(B^*,A^*)$ is exponentially detectable. For more on their meaning in control theory see 
Section \ref{S6}. 

We will now state our theorems for Lyapunov and Riccati equations. In the following we abuse notation by identifying operators  $A^*P+PA$ with their closures.
\begin{theorem}\label{Wonh} Let $X$ be a Banach space and $T(t)$ be a $C_0$ semigroup on it 
with the generator $A$. If the pair $(C,A)$ is exponentially detectable then the following conditions are equivalent:

{\rm(i)} $A^*P+PA=-C^*C$ has a positive definite solution $P\in\B_s^+(X,X^*)$ such that 
$P(\D_A)\subseteq\D_{A^*}$;

{\rm(ii)} $T(t)$ is exponentially stable;

{\rm(iii)} The Lyapunov generator $\L_A$ has a bounded inverse on $\B_s(X,X^*)$ and $-(\L_A)^{-1}\geq0$.
\end{theorem}
\noindent Recall that we defined "positive definite" by a non-strict inequality, so for $C=0$ the unique solution is $P=0$, and the theorem is vacuously true ($T(t)$ is exponentially stable by definition of $(0,A)$ being exponentially detectable). The exponential detectability above can be replaced with a weaker condition of detectability in $L^2$, see Section \ref{S6}. This theorem was proved in \cite{K2} for reflexive spaces, but the reflexivity assumption can be dropped using Theorem \ref{ComGen}.
\begin{theorem}\label{WonhBan} Let $A$ be the generator of a $C_0$ semigroup on  $X$, $U,V$ be Hilbert spaces, and 
$B\in\B(U,X)$, $C\in\B(X,V)$. Suppose that $(A,B)$ is exponentially stabilizable and $(C,A)$ is exponentially detectable. Then
 
{\rm(i)} The algebraic operator Riccati equation 
\begin{equation}\label{TheoRicc}
A^*P+PA-PBB^*P+C^*C=0
\end{equation}
has a unique positive solution $P$ such that $P(\D_A)\subseteq\D_{A^*}$, and
$\|A^*P+PA\|<\infty$. This solution is stabilizing, i.e. $A-BB^*P$ generates an exponentially stable semigroup .

{\rm(ii)} Pick $P_0$ so that $A-BB^*P_0$ generates an exponentially stable semigroup. This can always be done, e.g. by solving the Lyapunov equation
\begin{equation*}
(A-BK)^*P_{0}+P_{0}(A-BK)=-C^*C-K^*K,
\end{equation*}
with $A-BK$ generating an exponentially stable semigroup. Then the solution $P$ can be obtained  as a monotone strong operator limit of solutions to linear Lyapunov equations
\begin{equation}\label{LyapIter}
(A-BB^*P_n)^*P_{n+1}+P_{n+1}(A-BB^*P_n)=-C^*C-P_nBB^*P_n\,.
\end{equation}
Moreover, there is $\varkappa>0$ such that $\|P-P_{n+1}\|\leq\varkappa\|P-P_n\|^2$ (quadratic convergence).
\end{theorem} 
\noindent The approximations in part (ii) are obtained by formally applying Newton's method to the equation 
$F(P):=\L_A P+\Phi(P)=0$, and monotonicity is due to the fact that $F(P)$ is order concave, that is $F\big(\alpha P+(1-\alpha)R\big)\geq\alpha F(P)+(1-\alpha)F(R)$ for $\alpha\in[0,1]$. The proof of convergence relies on the properties of $C_0^*$ generators. In fact, we first prove an existence result in a general setting of order concave equations with linear parts generating positive $C_0^*$ semigroups, see Theorem \ref{NewtRicc}. Quadratic convergence, unfortunately, is not nearly as useful in infinite dimensional Banach spaces as in finite dimensional ones. The theorem only assures strong operator convergence of $P_n$ to $P$, so $\|P-P_n\|$ may never get smaller than $\frac1\kappa$. But only if that happens does quadratic convergence imply exponential rate of convergence by norm.

Note that that $B$ and $C$ enter \eqref{TheoRicc} only through $N:=BB^*\in\B_s^+(X^*,X)$ and $Q:=C^*C\in\B_s^+(X,X^*)$. One could restate the theorem using only these operators because they can always be canonically factored. For $Q$ for example define an inner product on $\text{Ran}(Q)\subset X^*$ by $(Qx,Qy):=\langle Qx,y\rangle$. This is well defined, if $Q\widetilde{x}=Qx$ and $Q\widetilde{y}=Qy$ the value is the same since $Q^*\big|_{X}=Q$. Denote by $V$ the completion of $\text{Ran}(Q)$ in the Hilbert norm, this $V$ is called the reproducing kernel Hilbert space \cite{GvN,vN2}. Set $Cx:=Qx$, to find the adjoint take $z=Qy\in\text{Ran}(Q)$ and compute
$$
(Cx,z)=(Qx,Qy)=\langle Qx,y\rangle=\langle x,Qy\rangle=\langle x,z\rangle\,.
$$
So $C^*$ acts on $\text{Ran}(Q)$ as the natural inclusion to $X^*$, and hence coincides with it on $V$ since the range is dense. Thus, $Q=C^*C$, and $C$ is canonically recovered from $Q$. When $X$ is a Hilbert space $C$ can be identified with the positive square root $Q^{\frac12}$, and $V$ with $\text{Ran}(Q^{\frac12})$. A similar construction factorizes $N$.
To illustrate the theorem we give a couple of examples.
\bigskip

\noindent {\bf Example 1:}
Any operator $N\in\B_s^+(X^*,X)$ defines multiplication on $\B(X,X^*)$ by $(P,R)\mapsto PNR$. Given also a $Q\in\B_s^+(X,X^*)$ we can talk about taking a "positive square root", i.e. solving $PNP=Q$ for $P\in\B_s^+(X,X^*)$. It is clear however that even for matrices this equation is not solvable for all pairs $N$ and $Q$. 

Consider instead a regularized equation $PNP+2aP=Q$ for some $a>0$. Written as $-2aP-PNP+Q=0$ it is of the 
form \eqref{TheoRicc} with $A=-aI$, where $I$ is the identity operator. Since this $A$ generates an exponentially stable semigroup for any $a>0$ the pair $(A,B)$ is exponentially stabilizable for any $B$, and $(C,A)$ is exponentially detectable for any $C$. Theorem \ref{TheoRicc} now implies existence of a unique stabilizing solution. One can take $P_0:=\frac1{2a}Q$ as the initial guess for the Newton's iteration. 

Note that for $a=0$ Wonham's theory does not apply in general even if $PNP=Q$ is solvable. Existence of a stabilizing solution implies that $-NP$ generates an exponentially stable semigroup, and therefore is invertible. Thus, a linear isomorphism of a Banach space would have to factor through a Hilbert space, a very special property.
\bigskip

\noindent {\bf Example 2:}
In many applications the control and the observation spaces are finite dimensional \cite{AAW}. For asymptotic stabilization to be possible the "unstable part" of the generator $A$ has to be finite-dimensional as well. To make this more precise assume that $A$ is {\sl exponentially dichotomous}, i.e. the space $X$ is a direct sum $X_-\dot{+}X_+$ of $T(t)$ invariant subspaces such that the restrictions $A_{-}:=A\big|_{X_-}$ and $A_{+}:=-A\big|_{X_+}$ generate exponentially stable semigroups 
\cite{RM}. It follows that the spectrum of $A$ does not meet the imaginary axis, and if $A$ is bounded this suffices for  the exponential dichotomy. 

Assume that $\dim(X_+)<\infty$ and $X_+\subseteq\Ran B$. Denoting $I_+$ the identity operator on $X_+$ we can use a basis in $X_+$ to define $K\in\B(X,U)$ such that $BK=0\,\dot{+}\,(A_++I_+)$. Then $A-BK=A_-\,\dot{+}\,(-I_+)$ generates an exponentially stable semigroup, and $(A,B)$ is exponentially stabilizable. Similarly, $(C,A)$ is exponentially detectable if $\Ker C\subseteq X_-$ since then we can define $L\in\B(V,X)$ such that $LC=0\,\dot{+}\,(A_++I_+)$. Thus, if the range and the kernel conditions are satisfied there is a unique stabilizing feedback control, a result widely used in Hilbert spaces.

Note that we actually need something weaker than the exponential dichotomy in this case, as long as $\dim(X_+)$ is finite $A_+$ can be allowed to have purely imaginary or even positive real part spectrum. If $\dim(X_+)=\infty$ however the range and the kernel conditions are not enough, and their replacements depend on conditions for factorization through Hilbert spaces \cite[Ch.2]{Pis}.

\section{Topologies and continuity}\label{S5}

This section develops some analytic tools for studying the operator Riccati equations, and proving convergence of approximations to their solutions. First, we introduce projective tensor products that are preduals to spaces of operators. Then we consider implemented semigroups, that restrict to the Lyapunov semigroups, and tensor product semigroups, that they are adjoint to. Their continuity in some natural operator topologies is discussed, and the generators are described. To handle the nonlinear part of the equations we then look into properties of positive cones in spaces of symmetric operators, and continuity of quadratic maps on monotone sequences in them. 

\begin{definition}\label{tenprod} Let $X$ and $Y$ be Banach spaces and $X\otimes Y$ be their algebraic tensor product. The  duality pairing between $\B(X,Y^*)$ and $X\otimes Y$ is defined by $\llangle x\otimes y,P\rrangle:=\langle Px,y\rangle$ on monomials, and extended by linearity. Given $\rho\in X\otimes Y$ its {\sl projective tensor norm} is defined by 
$$
\textstyle{\|\rho\|:=\inf\{\sum_i\|x_i\|\|y_i\|\,\Big|\,\rho=\sum_ix_i\otimes y_i,\,x_i\in X,\,y_i\in Y\}\,.}
$$
The {\sl projective tensor product $X\otimes_\pi Y$} is the completion of $X\otimes Y$ in this norm \cite[2.1]{Ry}. 
\end{definition}
\noindent The projective tensor product is convenient for several reasons. First, its elements can be described explicitly, they are of the form $\sum_{i=1}^\infty a_i\,x_i\!\otimes y_i$ with bounded sequences $x_i, y_i\in X$, and a summable numerical sequence $a_i\in\R$ (this is the Grothendieck representation theorem \cite[III.6.4]{Sch}). Second, since $\|x\otimes y\|=\|x\|\|y\|$ it is almost immediate that the pairing $\llangle\,\cdot\,,\,\cdot\,\rrangle$ extends from $X\otimes Y$ to 
$X\otimes_\pi Y$, giving us a Banach space isomorphism $(X\otimes_\pi Y)^*\simeq\B(X,Y^*)$. Moreover, the dual norm on $\B(X,Y^*)$ is exactly the induced operator norm, so this isomorphism is isometric \cite[2.1]{Ry}. Finally, the semigroups on $X\otimes_\pi Y$ that the Lyapunov semigroups are adjoint to also admit an explicit description, see Definition \ref{ImplTens}.

Now let us recall some natural topologies on operator spaces. Given two Banach spaces $X$ and $Y$ consider the space 
$\B(X,Y^*)$ of bounded operators from $X$ to the Banach dual of $Y$. Aside from the weak* topology induced by duality with the predual $X\otimes_\pi Y$ the following three topologies will be of use to us:
\begin{itemize}
\item Weak operator* (wo*) with seminorms $R\mapsto|\langle Rx,y\rangle|$ for $x\in X,y\in Y$;
\item Strong operator (so) with seminorms $R\mapsto\|Rx\|$ for $x\in X$;
\item Ultraweak* (uw*) with seminorms $R\mapsto|\sum_{i=1}^\infty\langle Rx_i,y_i\rangle|$ for $x_i\in X,y_i\in Y$, and\\ 
$\sum_{i=1}^\infty\|x_i\|^2<\infty$, $\sum_{i=1}^\infty\|y_i\|^2<\infty$.
\end{itemize}
\noindent The weak operator* topology is obviously weaker than the other two, and the same argument as for Hilbert spaces \cite[I.3.2]{Dix} shows that it coincides with the ultraweak* topology on bounded subsets of $\B(X,Y^*)$. Recall that given a locally convex space $\X$ and a separating subspace $\mathcal{Y}$ of its algebraic dual, $\s(\X,\mathcal{Y})$ denotes the weakest topology on $\X$ in which all functionals from $\mathcal{Y}$ are continuous. It is specified by the seminorms 
$u\mapsto|\langle u,v\rangle|$, $v\in\mathcal{Y}$. With this notation we have the following.

\begin{lemma}\label{Ultra*} The weak operator* topology coincides with $\s(\B(X,Y^*),X\otimes Y)$, and the ultraweak* topology coincides with the weak* topology, $\s(\B(X,Y^*),X\otimes_\pi Y)$.
\end{lemma}
\begin{proof} The first claim is almost obvious from definitions since $\langle Rx,y\rangle=\llangle x\otimes y,R\rrangle$. For the second claim let $p(R):=|\sum_{i=1}^\infty\langle Rx_i,y_i\rangle|$ be 
an ultraweak* continuous seminorm, and set $\rho:=\sum_{i=1}^\infty x_i\otimes y_i$. Taking the projective tensor norm:
$$
\|\rho\|:=\|\sum_{i=1}^\infty x_i\otimes y_i\|\leq\sum_{i=1}^\infty\|x_i\otimes y_i\|
=\sum_{i=1}^\infty\|x_i\|\|y_i\|
\leq\Big(\sum_{i=1}^\infty\|x_i\|^2\Big)^{\frac12}\Big(\sum_{i=1}^\infty\|y_i\|^2\Big)^{\frac12}<\infty.
$$
Thus, $\rho\in X\otimes_\pi Y$ and $p(R)=|\llangle\rho,R\rrangle|$ is a $\s(\B(X,Y^*),X\otimes_\pi Y)$ seminorm.

Conversely, let $\rho\in X\otimes_\pi Y$. By the Grothendieck representation theorem $\rho=\sum_{i=1}^\infty a_i\,u_i\otimes v_i$ with $\|u_i\|,\|v_i\|\leq M<\infty$ and $\sum_{i=1}^\infty|a_i|<\infty$ \cite[III.6.4]{Sch}. Set $x_i:=\text{sign}(a_i)|a_i|^{\frac12}u_i$ and $y_i:=|a_i|^{\frac12}v_i$, then $\sum_{i=1}^\infty\|x_i\|^2<\infty$, $\sum_{i=1}^\infty\|y_i\|^2<\infty$, and $|\llangle\rho,R\rrangle|=|\sum_{i=1}^\infty\langle Rx_i,y_i\rangle|$ is an ultraweak* seminorm. Thus the ultraweak* and the weak* topologies share the same continuous seminorms, and hence coincide.
\end{proof}
There are two closely related classes of semigroups on $\B(X,Y^*)$ and $X\otimes_\pi Y$ induced by semigroups of bounded operators on $X$ and $Y$, see \cite[I.3.16]{EN}, \cite[3.4]{Kuh0} and \cite{Fr}. 
\begin{definition}\label{ImplTens} Let $X$ and $Y$ be Banach spaces and $T(t)$ and $S(t)$ be $C_0$ semigroups on $X$ and $Y$ respectively. The semigroup $\T(t)P:=S^*(t)\,P\,T(t)$ on $\B(X,Y^*)$ will be called the {\sl semigroup implemented by $T(t)$ and $S(t)$}. For a topology $\tau$ on $\B(X,Y^*)$ we define the $\tau$-generator of $\T(t)$ by $\A^\tau P:=\tau\text{-}\lim_{t\to0}\frac1t(\T(t)P-P)$ on the domain $\D_{\A^\tau}$ where the limit exists. The semigroup defined by $\T_*(t)(x\otimes y):=S(t)x\otimes T(t)y$ and extended by linearity and continuity to $X\otimes_\pi Y$ is called the {\sl tensor product semigroup} induced by $S(t)$ and $T(t)$. 
\end{definition}
\noindent As the notation indicates $\T(t)$ is adjoint to $\T_*(t)$ \cite[Thm.4]{K}, and therefore is a weak* continuous or $C_0^*$ semigroup on $\B(X,Y^*)$. By Lemma \ref{Ultra*} it is also ultraweak* and hence weak operator* continuous. Since the last two topologies coincide on bounded subsets we have $\A^{\text{uw*}}=\A^{\text{wo*}}$, and this generator is weak* densely defined and weak* closed. In general, $S^*(t)$ may only be a $C_0^*$ semigroup, but if $Y$ is reflexive it is a $C_0$ semigroup on 
$Y^*$ \cite[3.1.8]{BR}, \cite[IX.1]{Yo}. In that case $\T(t)$ is also strong operator continuous. Indeed, 
\begin{align*}
\|\bigl(&\T(t+h)P-\T(t)P\bigr)x\|\\
&\leq\|S^*(t+h)\|\,\|P\|\,\|T(t+h)x-T(t)x\|+\|\bigl(S^*(t+h)-S^*(t)\bigr)PT(t)x\|
\xrightarrow[h\to0]{}0\,.
\end{align*}
Moreover, $\T(t)$ is locally uniformly bounded in the sense of \cite{K}, and locally bi-continuous in the sense of K\"uhnemund \cite[Prop.3.16]{Kuh0}. Either property implies that $\A^{\text{so}}$ is strong operator densely defined, and strong operator sequentially continuous. 

As observed by Freeman \cite{Fr}, when $Y$ is reflexive the strong operator generator   coincides with the weak operator one, and can be described explicitly. We now prove a more general result that applies to non-reflexive spaces, and also describes the generator in terms of the duality pairing.
\begin{theorem}\label{ComGen} Let $\T(t)$ be a semigroup on $\B(X,Y^*)$ implemented by $C_0$ semigroups $T(t)$ and $S(t)$ on 
$X$ and $Y$ respectively, with generators $A$ and $B$ respectively. Define $\A P:=\overline{B^*P+PA}$ (the bar stands for closure) on the domain
$$
\D_\A:=\{P\in\B(X,Y^*)\,|\,P(\D_A)\subseteq\D_{B^*},\,\|B^*P+PA\|<\infty\}\,.
$$
Then,

{\rm(i)} $P\in\D_\A$ if and only if there exists a bounded operator $Q\in\B(X,Y^*)$ such that\\
$\langle Px,By\rangle+\langle PAx,y\rangle=\langle Qx,y\rangle$ for all $x\in\D_A$ and $y\in\D_B$, in which case 
$\A P=Q$;

{\rm(ii)} $\A^{\text{wo*}}=\A$;

{\rm(iii)} If additionally $Y$ is reflexive then $\A^{\text{so}}=\A$;
\end{theorem}
\begin{proof}{\rm(i)} Assume that such $Q$ exists. Let $x\in\D_A$ and 
$z:=Px$, then for any $y\in\D_B$ the definition of $Q$ yields
$$
\langle z,By\rangle=-\langle PAx,y\rangle+\langle Qx,y\rangle=\langle(-PA+Q)x,y\rangle\,.
$$
By definition of the adjoint it follows that $z=Px\in\D_{B^*}$ and $B^*(Px)=-PAx+Qx$. Therefore, $P(\D_A)\subseteq\D_{B^*}$ and 
$\|B^*P+PA\|=\|Q\|<\infty$. Thus, $P\in\D_\A$ and $\A P=Q$. 

Conversely, if $P\in\D_\A$ then $\A P$ is bounded and $\A P=B^*P+PA$ on $\D_A$. Taking $x\in\D_A$ and $y\in\D_B$ we have 
$\langle(\A P)x,y\rangle=\langle Px,Gy\rangle+\langle PAx,y\rangle$, so $Q:=\A P$ exists.

{\rm(ii)} First we prove that $\A\subseteq\A^{\text{wo*}}$. Let $P\in\D_\A$ and $x\in\D_A$, $y\in Y$, then
\begin{align*}
&\frac1h\Big\langle\Bigl(\T(t+h)P-\T(t)P\Bigr)x,y\Big\rangle
=\frac1h\Big\langle\Bigl(S^*(t+h)PT(t+h)-S^*(t)PT(t)\Bigr)x,y\Big\rangle\\
&=\Big\langle S^*(t+h)P\frac1h\Bigl(T(t+h)-T(t)\Bigr)x,y\Big\rangle+\Big\langle\frac1h\Bigl(S^*(t+h)-S^*(t)\Bigr)PT(t)x,y\Big\rangle\\
&=\Big\langle\frac1h\Bigl(T(t+h)-T(t)\Bigr)x,P^*S(t+h)y\Big\rangle+\Big\langle\frac1h\Bigl(S^*(t+h)-S^*(t)\Bigr)PT(t)x,y\Big\rangle\,.\\
\end{align*}
Since $x\in\D_A$ also $T(t)x\in\D_A$, and in the first term $\frac1h\bigl(T(t+h)-T(t)\bigr)x\xrightarrow[h\to0]{}AT(t)x$ by norm. Moreover, $P^*S(t+h)y\xrightarrow[h\to0]{}P^*S(t)y$ by norm since $S(t)$ is a $C_0$ semigroup. Therefore, as $h\to0$ the first term converges to $\langle T(t)Ax,P^*S(t)y\rangle=\langle S^*(t)PT(t)Ax,y\rangle$. Since $T(t)x\in\D_A$ and $P\in\D_\A$ we have $PT(t)x\in\D_{B^*}$. Hence the second term converges to $\langle S^*(t)B^*PT(t)Ax,y\rangle$. Summarizing we see that 
$$
\frac{d}{dt}\langle\T(t)Px,y\rangle
=\langle S^*(t)PAT(t)x,y\rangle+\langle S^*(t)B^*PT(t)x,y\rangle\\
=\langle S^*(t)\big(B^*P+PA\big)T(t)x,y\rangle\,.
$$
Integrating both sides from $0$ to $h$ we obtain,
\begin{equation}\label{LyapGenInt}
\langle(\T(h)P-P)x,y\rangle=\int_0^h\langle S^*(t)\big(B^*P+PA\big)T(t)x,y\rangle\,dt.
\end{equation}
Since $P\in\D_\A$ the operator $B^*P+PA$ is bounded on $\D_A$, and it has a unique bounded extension to $X$, which is its closure $\overline{B^*P+PA}$. The equality in \eqref{LyapGenInt} extends to all $x\in X$ if $B^*P+PA$ is replaced by $\overline{B^*P+PA}$. Dividing both sides by $h$ and passing to limit as $h\to0$ we arrive at $\langle(\A^{\text{wo*}}P)x,y\rangle=\langle(\overline{B^*P+PA})x,y\rangle$ for all $x\in X$, $y\in Y$. So $\D_\A\subseteq\D_{\A^{\text{wo*}}}$ and 
$\A^{\text{wo*}}P=\A P$ on $\D_\A$, i.e. $\A\subseteq\A^{\text{wo*}}$.

Now we prove that $\A^{\text{wo*}}\subseteq\A$. Take $P\in\D_{\A^{\text{wo*}}}$ and consider the elementary identity
$$
\frac1h\Bigl(S^*(h)-I\Bigr)P=\frac1h\Bigl(\T(h)P-P\Bigr)-S^*(h)\,P\,\frac1h\Bigl(T(h)-I\Bigr)\,.
$$
Applying both sides to $x\in\D_A$ and taking the weak limit in $X$ as $h\to0$ we obtain
on the right $(\A^{\text{wo*}} P)x-PAx$\,. Therefore, $Px\in\D_{\B^*}$ and $B^*Px=(\A^{\text{wo*}} P)x-PAx$. Hence  
$P(\D_A)\subseteq\D_{B^*}$ and $\A^{\text{wo*}} P=B^*P+PA$ on $\D_A$. Since $\A^{\text{wo*}}P$ is bounded so is $B^*P+PA$, and since $\D_A$ is dense $\A^{\text{wo*}} P=\overline{B^*P+PA}$ everywhere. Thus, $P\in\D_\A$ and $\A^{\text{wo*}}=\A$ on 
$\D_{\A^{\text{wo*}}}$. Combining the two inclusions we have $\A^{\text{wo*}}=\A$.

{\rm(iii)} When $Y$ is reflexive the weak and the weak* topologies on it coincide, $S^*(t)$ is a $C_0$ semigroup, and its weak generator $B^*$ is the same as its strong generator by a theorem of Yosida \cite[IX.1]{Yo}. Hence we immediately have from (ii) that $\A^{\text{so}}\subseteq\A^{\text{wo*}}\subseteq\A$.

For the converse inclusion note that $x\in\D_\A$ still implies $PT(t)x\in\D_{B^*}$ and \eqref{LyapGenInt} holds. But now $S^*(t)\big(B^*P+PA\big)T(t)x$ is norm continuous, and $(\T(h)P-P)x=\int_0^h S^*(t)\big(B^*P+PA\big)T(t)x\,dt$ is a Bochner integral. Reasoning as in (ii) we now get $\A\subseteq\A^{\text{so}}$ and $\A=\A^{\text{so}}$.
\end{proof}
\noindent When $Y$ is not reflexive $\T(t)$ may not be strong operator continuous, and $\D_{\A^{\text{so}}}$ may be strictly smaller than $\D_\A$. In other words, for some $P\in\D_\A$ the difference quotient $\frac1h(\T(h)P-P)$ may not strong operator converge to $\overline{B^*P+PA}$. 

For the Lyapunov semigroups $X=Y$, $T(t)=S(t)$, and we restrict $\T(t)$ to the subspace of symmetric operators $\B_s(X,X^*)$, a predual to which is the subspace of symmetric tensors.
\begin{definition}\label{symtenprod}
The algebraic symmetric tensor product $X\ohat X$ is the linear span in $X\otimes X$ of tensors of the form $x\otimes y+y\otimes x$, where $x,y\in X$, and the {\sl projective symmetric tensor product $X\ohat_\pi X$} is the closure of $X\ohat X$ in $X\otimes_\pi X$.
\end{definition}
\noindent The projective duality $(X\otimes_\pi Y)^*\simeq\B(X,Y^*)$ \cite[2.1]{Ry}, and the standard formulas for isomorphisms of subspaces and quotients \cite[4.8]{Rud} show that also $(X\ohat_\pi X)^*\simeq\B_s(X,X^*)$. By inspection, Theorem \ref{ComGen} remains true for the restricted semigroups without change. It then follows from Theorem \ref{ComGen}(i) that the duality pairing Riccati equation \eqref{DualParRicc} is equivalent to the Lyapunov generator equation 
$\L_AP+\Phi(P)=0$ in the sense that any solution to one is also a solution to the other. The following property of Lyapunov generators is immediately obvious from their explicit description in the theorem, and will be used in Section \ref{S6}.
\begin{corollary}\label{LyapAdd} Let $X$ be a Banach space, $A$ be a generator of a $C_0$ semigroup on $X$, and $\L_A$ be the corresponding Lyapunov generator. Then for any and $G\in\B(X)$ we have $\L_{A+G}=\L_A+\L_G$, in particular $\D_{\L_{A+G}}=\D_{\L_{A}}$.
\end{corollary}

Our second concern is the continuity of the quadratic part of the Riccati equation. To this end we investigate relations between topology and order in $\B_s(X,X^*)$. When $X=H$ is a Hilbert space $\B_s(X,X^*)=\B_s(H)$ is the familiar space of bounded self-adjoint operators with the cone of positive definite operators defining the partial order. As in the Hilbert case $\B_s(X,X^*)$ is isomorphic to the space of bounded quadratic (or bilinear symmetric) forms on $X$ \cite[2.2]{Ry}. For the quadratic forms the partial order is just the pointwise order. However, in some respects the partial order on $\B_s(X,X^*)$ may behave quite differently than on $\B_s(H)$. 

For one, unless $X$ is isomorphic to a Hilbert space, $\B_s(X,X^*)$ contains no operators that multiplied by a large enough number become larger than any given operator, like the identity $I\in\B_s(H)$. One can show that any such operator $E$ must satisfy $\langle Ex,x\rangle\geq a\|x\|^2$ for some $a\geq0$, and therefore $\langle Ex,x\rangle^{\frac12}$ is an equivalent Hilbert norm on $X$. 

Moreover, $\B_s^+(X,X^*)$ may not be generating, i.e. not every bounded and symmetric operator can be represented as a difference of two positive definite operators. The authors of \cite{Kal} prove that $\B_s^+(X,X^*)$ is generating if and only if every $P\in\B_s(X,X^*)$ factors through a Hilbert space (by the reproducing kernel Hilbert space construction any $P\in\B_s^+(X,X^*)$ always so factors, see \cite[2.1]{GvN} and Section \ref{S4}). An example of non-factorizable $P\in\B_s(l_p,l_p^*)$ for $1<p<2$ is given in \cite{Sari}, hence $\B_s^+(l_p,l_p^*)$ is not generating for such $p$. Moreover, any infinite-dimensional $L_p(\mu)$ contains a complemented copy of $l_p$ \cite[Thm.3.3]{Kal}, so $\B_s^+(L_p(\mu),L_p(\mu)^*)$ is not generating for $1<p<2$ either. However, some properties of the order are shared with $\B_s(H)$, as the next theorem shows.
\begin{theorem}\label{Ls+} Let $\|\cdot\|$ denote the induced operator norm on $\B(X,X^*)$, then  

{\rm(i)} If $P\in\B_s(X,X^*)$ then $\ds{\|P\|=\sup_{\|x\|=1}|\langle Px,x\rangle|}$, hence $\|\cdot\|$ is monotone on 
$\B_s^+(X,X^*)$;

{\rm(ii)} If $P\in\B_s^+(X,X^*)$ then $\|Px\|^2\leq\|P\|\,\langle Px,x\rangle$ for all $x\in X$;

{\rm(iii)} order bounded monotone sequences in $\B_s(X,X^*)$ strong operator converge.

\end{theorem}
\begin{proof}(i) The proof is analogous to the Hilbert space case, see e.g. \cite[Ch.VI]{RS}. Let $M:=\sup_{\|x\|=1}|\langle Px,x\rangle|$, obviously $M\leq\|P\|$. By symmetricity of $P$ and the parallelogram identity,
$$
4\langle Px,y\rangle=\langle P(x+y),x+y\rangle-\langle P(x-y),x-y\rangle
\leq|\langle P(x+y),x+y\rangle|+|\langle P(x-y),x-y\rangle|\,.
$$
Therefore,
$$
4\langle Px,y\rangle\leq M\,\bigr(\|x+y\|^2+\|x-y\|^2\bigl)\leq2\,M\,\bigr(\|x\|^2+\|y\|^2\bigl)\,.
$$
Cancel $2$ and set $y=\frac{\|x\|}{\|Px\|}Px$ in the last inequality to obtain
$$
2\Bigl\langle Px,\frac{\|x\|}{\|Px\|}Px\Bigr\rangle=2\|Px\|\,\|x\|
\leq M\,\Bigr(\|x\|^2+\frac{\|x\|^2}{\|Px\|^2}\left\|Px\right\|^2\Bigl)=2\,M\,\|x\|^2\,.
$$
Canceling $2\,\|x\|$ we get $\|Px\|\leq M\,\|x\|$ for all $x$, and hence the converse inequality $\|P\|\leq M$ holds. 

Recall that a norm is called monotone if $0\leq P\leq R$ implies $\|P\|\leq\|R\|$. But if $P\leq R$ then $\langle Px,x\rangle\leq\langle Rx,x\rangle$ for all $x\in X$ by definition of order. For positive definite $P$ the absolute value in the formula just proved can be omitted, so taking supremum over all $x$ with $\|x\|=1$ gives $\|P\|\leq\|R\|$.\\

\noindent (ii) This inequality is derived in passing in \cite{Dri}. Consider $\pfi\in X^*$, for any $\e>0$ there is a normalized element $\widetilde{g}\in X$ such that $\langle\pfi,\widetilde{g}\rangle\geq(1-\e)\|\pfi\|$ since $\ds{\|\pfi\|=\sup_{x\in X,\|x\|=1}\langle\pfi,x\rangle}$. Setting $g:=\|\pfi\|\widetilde{g}$ we have $\|g\|=\|\pfi\|$ and $\langle\pfi,g\rangle\geq(1-\e)\|\pfi\|^2$, in other words $g$ approximates in $X$ the tangent functional to $\pfi$ (which in general belongs to $X^{**}$).

Applying the last inequality to $\pfi=Px$ we have $(1-\e)\|Px\|^2\leq\langle Px,g\rangle$. Moreover, if $P\geq0$ then $\langle Px,y\rangle$ is a semi-definite inner product on $X$, so it satisfies the Cauchy-Schwarz inequality $|\langle Px,y\rangle|^2\leq\langle Px,x\rangle\,\langle Py,y\rangle$. Setting $y=g$ and combining the inequalities we obtain
\begin{align*}
(1-\e)^2\|Px\|^4&\leq|\langle Px,g\rangle|^2\leq\langle Px,x\rangle\langle Pg,g\rangle\\
&\leq\|P\|\,\|g\|^2\,\langle Px,x\rangle=\|P\|\,\|Px\|^2\,\langle Px,x\rangle.
\end{align*}
Canceling $\|Px\|^2$ we get $(1-\e)^2\|Px\|^2\leq\|P\|\,\langle Px,x\rangle$ (if $Px=0$ this inequality holds trivially), and since $\e>0$ is arbitrary the desired inequality follows.

\noindent (iii) For any $x\in X$ the sequence $\langle P_n x,x\rangle$ is monotone and bounded, and hence converges by the Weierstrass theorem. By the polarization identity a limit $l_x(y)$ of $\langle P_n x,y\rangle$ exists for any $x,y\in X$. 
Multiplying $P_n$ by $-1$ and adding a lower bound to all of them if necessary we may assume that $P_n$ are positive definite and monotone increasing. Since by (i) the operator norm is monotone the norms of $P_n$ are uniformly bounded, $\|P_n\|\leq M$. This implies that $l_x$ is a continuous functional on $X$, and one can define a linear operator from $X$ to $X^*$ by setting 
$Px:=l_x$. 

By construction, $P_n$ weak operator* converge to $P$. Moreover, by (ii)
\begin{align*}
\|(P_n-P)x\|^2\leq\|P_n-P\|\,\langle(P_n-P)x,x\rangle
\leq 2M\langle(P_n-P)x,x\rangle\xrightarrow[n\to\infty]{}0\,.
\end{align*}
Thus, $P_n$ strong operator converge to $P$.
\end{proof}
Operators in $\B_s(X,X^*)$ can not be composed directly, but any operator $N\in\B_s(X^*,X)$ can mediate composition and define a product on $\B_s(X,X^*)$ by $(P,R)\mapsto PNR$. Since $\|(P_nNR_n-PNR)x\|\leq\|P_n\|\|N\|\|(R_n-R)x\|+\|(P_n-P)NRx\|$ this product is sequentially strong operator continuous. Therefore, we have the following corollary for the nonlinear part of the Riccati equations.
\begin{corollary}\label{PNR} For any $N\in\B_s(X^*,X)$ the product on $\B_s(X,X^*)$ defined by $(P,R)\mapsto PNR$  is continuous on order bounded monotone sequences.
\end{corollary}

\section{Concave equations and stabilizing solutions}\label{S3}

We now have the analytic tools to construct solutions to the Riccati equations. However, it turns out that the construction mostly uses order properties, and can be abstracted from the Riccati specifics. What we need is that the equations are quasi-linear with the linear part generating a positive semigroup, and with the non-linear part being concave. This is the generality adopted in this section.

We start by recalling some basic facts about ordered Banach spaces, positive $C_0^*$ semigroups on them, and the order concavity. We also recall some key results of \cite{K2} on abstract Lyapunov equations. They are then used to construct a special class of solutions to quasi-linear concave equations by Newton's iteration. In dimension one Newton's iteration provides monotone convergence of approximations for concave equations not just locally, but globally. We prove a similar 
result in ordered Banach spaces. 

Let $\X$ be a Banach space with predual $\X_*$ partially ordered by a closed proper cone $\X^+$, see \cite[App.2]{Clem}. The predual cone is defined in the usual way $\X_*^+:=\{\pfi\in \X_*\,|\, \langle\pfi,x\rangle\geq0\text{ for all }x\in \X^+\}$. Recall that a $C_0^*$ semigroup $\T(t)$ on $\X$ is a weak* continuous semigroup adjoint to a $C_0$ semigroup 
$\T_*(t)$ on $\X_*$. Its generator $\A$ is the weak* limit of the difference quotients, which is weak* densely defined, weak* closed, and adjoint to the generator $\A_*$ of the predual semigroup $\T_*(t)$ \cite[3.1.2]{BR}. If $\T(t)$ is positive, i.e. preserves $\X^+$, then so is $\T_*(t)$, and vice versa.

We will be studying quasi-linear equations of the form $\A x+\Phi(x)=0$ on $\X^+$, where $\A$ is the generator of a positive $C_0^*$ semigroup, and $\Phi$ is Gateaux differentiable \cite[ch.XVII.1.1]{Kan},\cite[I.2.1]{Vain}, and concave in some sense. Algebraic operator Riccati equations are not the only ones of this form, so are some reaction-diffusion equations in mathematical chemistry for example. With the help of Newton's iteration the problem reduces to solving a sequence of linear equations of the form $\A x+z=0$. Our assumptions about $\A$ and $z$ turn them into abstract analogs of the operator Lyapunov equations. The main technical tool is a criterion that relates solvability of $\A x=-z$ on $\X^+$ to positive invertibility of $-\A$, and to exponential stability of $\T(t)$, if $z$ satisfies a non-degeneracy condition below.
\begin{definition}\label{L1Det} An element $z\in\X^+$ is called an {\sl $L^1$ detector for $\T(t)$ on a subset $\F\subseteq\X_*^+$} if for every $\pfi\in \F$: 
\begin{equation*}\label{L1det}
\int_0^\infty\langle\T_*(t)\pfi,z\rangle\,dt<\infty\implies\int_0^\infty\|\T_*(t)\pfi\|\,dt<\infty\,.
\end{equation*}
A subset $\F\subset \X_*^+$ is called a {\sl stability subset} for a class of semigroups if for every semigroup in the class 
$\int_0^\infty\|\T_*(t)\pfi\|\,dt<\infty$ for all $\pfi\in\F$ implies that the same inequality holds for all $\pfi\in\X_*^+$. 
\end{definition}
\noindent Basically, the detector condition asks that $z$ correctly "detect" asymptotic behavior of $\T_*(t)$ on every element of $\F$. The terminology comes from examples in control theory, see Section \ref{S6}. Readers familiar with order unit 
spaces \cite[App.2]{Clem} will immediately see that any order unit is an $L^1$ detector on all of $\X_*^+$ for any semigroup. But semigroups can have $L^1$ detectors even in spaces without order units. It is also immediate from the definition that if $w\geq z$ and $z$ is an $L^1$ detector, then so is $w$. Note that $L^1$ stability of $\T_*(t)$ on $\F$ does not in general imply $L^1$ stability on the entire $\X_*^+$, even if $\F$ is dense in it. Of course, $\F=\X_*^+$ is a stability subset for all $\C_0^*$ semigroups, but the class of Lyapunov semigroups admits a much smaller stability subset, see Lemma \ref{StabSub}.

We will assume that $\X_*^+$ is generating, i.e. every element in $\X_*$ is a difference of two positive ones. We do not however impose this condition on $\X^+$ keeping in mind our example of $\X=\B_s(X,X^*)$ and $\X_*=X\ohat_\pi X$. If $\int_0^\infty\|\T_*(t)\pfi\|\,dt<\infty$ holds for every $\pfi\in\X_*^+$ and $\X_*^+$ is generating, then it holds for every $\pfi\in\X_*$. Then by the Datko-Pazy theorem \cite[II.1.2.2]{Ben}, \cite[Prop.9.4]{Clem} we have $\|\T_*(t)\|=\|\T(t)\|\leq Me^{-\e t}$ for $M,\e>0$ since $\T_*(t)$ is a $C_0$ semigroup, i.e. $\T(t)$ is exponentially stable. This makes the following result less surprising. We state it for $C_0^*$ semigroups, which is all we need here, but in \cite[Thm.3]{K2} it is proved in greater generality. 
\begin{theorem}\label{Lyappi} Let $\T(t)$ be a positive $C_0^*$ semigroup on $\X$ with the generator $\A$ and an $L^1$ detector $z$ on a stability subset $\F$. If $\X_*^+$ is generating the following conditions are equivalent:

{\rm(i)} $\A x=-z$ has a positive solution $x\in\D_\A\cap \X^+$;

{\rm(ii)} $\T(t)$ is exponentially stable;

{\rm(iii)} $\A$ has a bounded inverse on $\X$ and $-\A^{-1}\geq0$.
\end{theorem}
\noindent This is an abstract version of Wonham's stability criterion for Lyapunov equations. The most non-trivial implication is from $\rm(i)$ to  $\rm(ii)$. When using the theorem the hardest part to verify is the detector condition. One way is to use continuous final observability that implies $L^1$ detectability \cite[Thm.2]{K2}, and can often be established in applications by priori estimates, see e.g. \cite[3.27]{CP}.

For matrices there is a notion of matrix concavity/convexity \cite{Ando}, and it was used implicitly by Kleinman in the original proof of monotone convergence of approximations \cite{Kl} to solutions to algebraic matrix Riccati equations. Its use is made explicit for a larger class of matrix equations in \cite{Damm}, and Zabczyk generalized Kleinman's implicit approach to operator Riccati equations in Hilbert spaces \cite{Zb0}. Ordinary and matrix concavity are particular cases of order concavity defined next.
\begin{definition}\label{OrdConc} Let $\X$ be an ordered Banach space. A map $\Phi:\X\supseteq D\to\X$ is called {\sl order concave} if for all $\alpha\in[0,1]$ and $x,y\in\X^+$, we have $\Phi(\alpha x+(1-\alpha)y)\geq\alpha\Phi(x)+(1-\alpha)\Phi(y)$.
\end{definition}
\noindent If $D$ is a convex set (in the usual sense) and $\Phi$ is Gateaux differentiable one can prove as in the one dimensional case that order concavity is equivalent to $\Phi(y)-\Phi(x)\leq\Phi'(x)(y-x)$ for all $x,y\in D$, a condition which is often simpler to verify. The Riccati map $\Phi(P)=-PNP+Q$ is Gateaux differentiable, and when $N\geq0$, order concave on $\B_s(X,X^*)$, see Section \ref{S6}.

In the rest of this section $\A$ denotes the generator of a positive $C_0^*$ semigroup on $\X$ with the domain $\D_\A$, and $\Phi:\X^+\to\X$ is an order concave Gateaux differentiable map. We set $F(x):=\A x+\Phi(x)$ for $x\in\D_\A$, and for all $x,y\in\X^+$ we denote for convenience:
\begin{align}\label{F'LPsi}
F'(x)&:=\A+\Phi'(x)\notag\\ L(x)&:=\Phi(x)-\Phi'(x)x\\ \Psi(x,y)&:=\Phi(x)-\Phi(y)+\Phi'(x)(y-x).\notag
\end{align}
One can see that for $x\in\D_\A$ changing $\Phi$ to $F$ in the definitions of $L$ and $\Psi$ does not change the result. Here $\Psi$ measures deviation of $\Phi$ from its tangent, and in dimension one $L(x)=\Phi(0)+\Psi(x,0)$ is closely related to the Legendre transform of $F$ and $\Phi$. The order concavity is equivalent to $\Psi(x,y)\geq0$ for all $x,y\in D$, so if $0\in D$ then $L(x)\geq\Phi(0)$ for all $x\in D$.

Newton's iteration is a classical method for solving nonlinear equations of the form $F(x)=0$, see e.g. \cite[ch.XVIII]{Kan}. By construction, the iteration equation is $F(x_n)+F'(x_n)(x_{n+1}-x_n)=0$. It can be rewritten as
$F'(x_n)x_{n+1}=-L(x_{n+1})$, which is of the form $\A x=-z$. The next lemma describes the effect of performing one step of Newton's iteration in our setting.
\begin{lemma}[Iteration step]\label{Iter} Assume that $\X_*$ is generating and $F'(u):=\A+\Phi'(u)$ generates a positive $C_0^*$ semigroup $\T_{F'(u)}(t)$ for every $u\in\X^+$. Suppose $\Phi(0)\geq0$ and for some $x\in\X^+$ the semigroup $\T_{F'(x)}(t)$ is exponentially stable. Then:

{\rm(i)} There is a unique solution $y\in\D_\A\cap\X^+$ to $F'(x)y=-L(x)$;

{\rm(ii)} $F(y)=-\Psi(x,y)\leq0$;

{\rm(iii)} If moreover $y\in\D_\A\cap\X^+$ and $F(x)\leq0$, then $y\leq x$;

{\rm(iv)} If moreover $L(y)$ is an $L^1$ detector for $\T_{F'(y)}(t)$ then $\T_{F'(y)}(t)$ is exponentially stable.
\end{lemma}
\begin{proof}{\rm(i)} By assumption we have $L(x)\geq\Phi(0)\geq0$. Since $F'(x)$ is stable $-F'(x)^{-1}\geq0$ by Theorem \ref{Lyappi}, and $y=-F'(x)^{-1}L(x)\in\D_{F'(x)}\cap\X^+=\D_\A\cap\X^+$.

{\rm(ii)} By definition of $F$ and the equation for $y$ we compute $F(y)=F'(x)y+L(x)-\Psi(x,y)=-\Psi(x,y)\leq0$. 

{\rm(iii)} Since $x\in\D_\A$ the equation for $y$ is equivalent to $F(x)+F'(x)(y-x)=0$. And since $-F'(x)^{-1}\geq0$ by Theorem \ref{Lyappi}, $y-x=-F'(x)^{-1}F(x)\leq0$.

{\rm(iv)} Setting $y=x$ in the identity from {\rm(ii)} we obtain  $F'(y)y=-(L(y)-F(y))$. This is an abstract Lyapunov equation to which $y$ is a solution. Since $-F(y)\geq0$ by {\rm(ii)} and $L(y)$ is an $L^1$ detector $L(y)-F(y)$ is also an $L^1$ detector. Since $y\geq0$ our equation has a positive solution fulfilling part {\rm(i)} of Theorem \ref{Lyappi}. Part {\rm(iii)} of that theorem gives the desired conclusion.
\end{proof}
One can see that in conditions of Lemma  \ref{Iter} the assumptions imposed on $x$ are reproduced and even improved for the next approximation $y$. Therefore, assuming a suitable initial guess $x_0$ can be found, we can generate a monotone decreasing positive sequence $x_n$ that (hopefully) converges to a solution to $F(x)=0$. One kind of convergence already follows from our assumptions. If $\X_*$ is generating then monotone bounded sequences always weak* converge in $\X$. Indeed, numerical sequences $\langle\pfi,x_n\rangle$ converge by the Weierstrass theorem for any $\pfi\in\X_*^+$, and therefore for any $\pfi\in\X_*$. So $x_n$ is weak* Cauchy. Since bounded sets in the dual to a Banach space are weak* precompact by the Alaoglu theorem we conclude that $x_n$ weak* converge to some $x_\infty\in\X^+$. 

However, weak* topology is so weak that nonlinear maps of interest are rarely continuous in it. Fortunately, monotone sequences can often be proved to converge in a stronger sense, in which the relevant maps are continuous, see Corollary \ref{PNR}.
\begin{definition}\label{OrdConc} A map $\Theta:\X\supseteq D\to\X$ is called {\sl continuous on bounded monotone sequences} if for any bounded monotone sequence $x_n$ with weak* limit $x$ the values $\Theta(x_n)$ converge to $\Theta(x)$ as $n\to\infty$.
\end{definition}
To simplify the statement of the following theorem it is convenient to introduce some terminology. The motivation comes from control theory and is explained in Section \ref{S6}.
\begin{definition}\label{StabDet} We say that an element $x\in\X^+$ {\sl stabilizes} $(\A,\Phi)$ if the operator 
$F'(x):=\A+\Phi'(x)$ generates an exponentially stable positive $C_0^*$ semigroup. A pair $(\A,\Phi)$ is called 
{\sl exponentially stabilizable} if some $x\in\X^+$ stabilizes it. A pair $(\A,\Phi)$ is called {\sl globally $L^1$ detectable} if for every $x\in\X^+$ the element $L(x):=\Phi(x)-\Phi'(x)x$ is an $L^1$ detector for $\T_{F'(x)}(t)$ 
on a stability subset $\F$ (we suppress the dependence on $\F$ in the name). 
\end{definition}
The next theorem is an abstract version of our main result.
\begin{theorem}\label{NewtRicc} Let a pair $(\A,\Phi)$ be exponentially stabilizable and globally $L^1$ detectable, 
$\Phi(0)\geq0$ and $x\mapsto\Phi(x)$, $(x,y)\mapsto\Phi'(x)y$ be continuous on order bounded monotone sequences. Pick 
$x_0\in\X^+$ so that $\T_{F'(x_0)}(t)$ is exponentially stable. Then the sequence defined recursively by $F(x_n)+F'(x_n)(x_{n+1}-x_n)=0$ is well-defined, belongs to $\D_A\cap\X^+$, and is monotone decreasing for $n\geq1$. It weak* converges to a unique positive solution to $F(x)=0$, and this solution is stabilizing.
\end{theorem}
\begin{proof}\underline{Existence:} By solving the iteration equation and applying Lemma \ref{Iter} we produce $x_n\in\D_A\cap\X^+$ with $x_1\geq x_2\geq\dots\geq0$. The iteration equation can be rewritten as $\A x_{n+1}=-\Phi(x_n)+\Phi'(x_n)x_n-\Phi'(x_n)x_{n+1}.$
Since $x_n$ is a bounded monotone sequence there is a weak* limit $x_n\xrightarrow[n\to\infty]{}x_\infty\in\X^+$. By continuity assumptions on $\Phi$: 
$$
\A x_{n+1}\xrightarrow[n\to\infty]{}-\Phi(x_\infty)+\Phi'(x_\infty)x_\infty-\Phi'(x_\infty)x_\infty=-\Phi(x_\infty).
$$
Since $\A$ generates a $C_0^*$ semigroup it is weak* closed, so $x_\infty\in\D_A$ and $\A x_\infty=-\Phi(x_\infty)$.
Thus, $F(x_\infty)=\A x_\infty+\Phi(x_\infty)=0$, and $x_\infty$ is a solution.
\medskip

\underline{Stabilization:} Let $x\in\D_A\cap\X^+$ be any positive solution to $F(x)=0$, then 
$F'(x)x=-(0-F'(x)x)=-(F(x)-F'(x)x)=-L(x)$. Since $L(x)$ is an $L^1$ detector for $F'(x)$ by global detectability, we conclude from Theorem \ref{Lyappi} that $\T_{F'(x_\infty)}(t)$ is exponentially stable. Thus $x$, and $x_\infty$ in particular, is stabilizing.
\medskip

\underline{Uniqueness:} For any positive solution $x$ we have by concavity $0=F(x)-F(x_\infty)\leq F'(x_\infty)(x-x_\infty)$. Since $\T_{F'(x_\infty)}(t)$ is positive and exponentially stable $-F'(x_\infty)^{-1}\geq0$. Therefore, $x-x_\infty\leq-F'(x_\infty)^{-1}0=0$. Switching the roles of $x$ and $x_\infty$ we also get $x_\infty-x\leq0$, and $x=x_\infty$.
\end{proof}
\noindent The condition $F(0)=\Phi(0)\geq0$ can be somewhat relaxed to $F(x)\geq0$ having a solution $\theta\in\X^+$ and a stabilizing $x_0$ existing with $x_0\geq\theta$. This case reduces to ours by passing to $\widetilde{F}(x):=A(x+\theta)+\Phi(x+\theta)$, which satisfies the conditions of Theorem \ref{NewtRicc}. If $\widetilde{x}_\infty$ is a solution to $\widetilde{F}(x)=0$ then $x_\infty:=\widetilde{x}_\infty+\theta$ is a solution to the original $F(x)=0$. In this case $x_\infty$ is still a unique stabilizing positive solution, but other positive solutions may exist. They are smaller than or incomparable to 
$\theta$. Since any of them can be chosen in place of $\theta$ we conclude that $x_\infty$ is also the largest positive solution to $F(x)=0$, and even to $F(x)\geq0$. Such characterization of the stabilizing solution is also valid for stochastic matrix Riccati equations \cite{Damm}.

\section{Stabilizability and detectability}\label{S6}

In this section we prove our main result by expressing the conditions of the abstract Theorem \ref{NewtRicc} explicitly for the Lyapunov semigroups, and then verifying them. We also recall stabilizability and detectability conditions from control theory, and explain their relation to the abstract versions from the previous section.

One of equivalent ways to define detectability \cite[3.6]{Wh} in finite dimensional spaces is to call a matrix pair $(C,A)$ detectable if for every $x$: $Ce^{tA}x\xrightarrow[t\to\infty]{}0$ implies $e^{tA}x\xrightarrow[t\to\infty]{}0$. In infinite dimensions different types of convergence to $0$ are no longer equivalent, so many generalizations are possible, ours is one of them, cf. \cite[II.1.2.2]{Ben}. 
\begin{definition}\label{detL2} Let $X$ be a Banach space and $T(t)$ be a $C_0$ semigroup on it with the generator $A$. Let $V$ be a Hilbert space and $C:X\to V$ be a bounded operator. The pair {\sl $(C,A)$ is called detectable in $L^2$} if for all
$x\in X$ 
$$
\int_0^\infty\|CT(t)x\|^2\,dt<\infty\implies\int_0^\infty\|T(t)x\|^2\,dt<\infty\,.
$$
\end{definition}
Recall that any Lyapunov semigroup is adjoint to $\T_*(t)(x\otimes x)=T(t)x\otimes T(t)x$. It follows that $\|\T(t)\|=\|\T_*(t)\|=\|T(t)\|^2$ and exponential stability of $T(t)$ is equivalent to that of $\T(t)$. A crucial observation is that $(X\ohat_\pi X)^+$ has a very simple stability subset for Lyapunov semigroups as a consequence of the Datko-Pazy theorem  \cite[II.1.2.2]{Ben}, \cite[Prop.9.4]{Clem}. The theorem states that for $C_0$ semigroups if $\int_0^\infty\|T(t)x\|^p\,dt<\infty$ for some $p>1$ and all $x\in X$ then $T(t)$ is exponentially stable.
\begin{lemma}\label{StabSub} The set of tensor squares $\F:=\{x\otimes x\big|\,x\in X\}$ is a $\T_*(t)$ invariant stability subset of $(X\ohat_\pi X)^+$ for the class of Lyapunov semigroups. Moreover, $\int_0^\infty\|\T_*(t)\rho\|\,dt<\infty$ for all $\rho\in\F$ implies even exponential stability of $\T_*(t)$. 
\end{lemma}
\begin{proof} The invariance is obvious from the definition of $\T_*(t)$. And if $\int_0^\infty\|\T_*(t)(x\otimes x)\|\,dt=\int_0^\infty\|T(t)x\otimes T(t)x\|\,dt=\int_0^\infty\|T(t)x\|^2\,dt<\infty$ for all $x\in X$ then by the Datko-Pazy theorem $T(t)$, and hence $\T_*(t)$, is exponentially stable.
\end{proof}
\noindent Setting $R:=C^*C$ we can rewrite the definition of $L^2$ detectability as 
\begin{equation}\label{RdetL2}
\int_0^\infty\llangle\T_*(t)(x\otimes x),R\rrangle\,dt<\infty\implies\int_0^\infty\|\T_*(t)(x\otimes x)\|^2\,dt<\infty\,.
\end{equation}
In other words, $L^2$ detectability of $(C,A)$ is equivalent to $R$ being an $L^1$ detector in the abstract sense of Definition \ref{L1Det} for the Lyapunov semigroup of 
$T(t)$ on the stability subset of tensor squares. We will also say that $R$ is an $L^2$ detector for the generator $A$ of $T(t)$.

But for $(\L_A,\Phi)$ to be globally $L^1$ detectable in the sense of Definition \ref{StabDet} we will need $L(P)=Q+PNP$ to be an $L^1$ detector of $\L_{A-NP}$ for every $P\in\B_s^+(X,X^*)$. Simple $L^2$ detectability of $(C,A)$ is not enough to ensure  detectability globally. One can show however that exponential detectability of $(C,A)$ is enough. A better alternative would be to prove that $L^2$ detectability is preserved by the Newton's iteration, but this seems unlikely to hold in general. 
\begin{lemma}\label{ExpDetL2} Let $A$ be the generator of a positive $C_0$ semigroup $T(t)$ on $X$, $B\in\B(U,X)$ and $C\in\B(X,V)$ for some Hilbert spaces $U,V$. Suppose that $(C,A)$ is exponentially detectable. Then for any $K\in\B(X,U)$ the operator 
$C^*C+K^*K$ is an $L^2$ detector for $A-BK$.
\end{lemma} 
The proof is analogous to the Hilbert case \cite[Lem.3]{Zb0}, see also \cite[Thm.I.1.2.6]{Ben}, and is omitted.
\begin{corollary}\label{GlobL1Det} Suppose $A$ is the generator of a $C_0$ semigroup on $X$, and $\Phi(P):=-PBB^*P+C^*C$ for $B\in\B(U,X)$, $C\in\B(X,V)$ with some Hilbert spaces $U,V$. If $(C,A)$ is exponentially detectable then $(\L_A,\Phi)$ is globally $L^1$ detectable on $\B_s^+(X,X^*)$.
\end{corollary} 
\begin{proof} By Definition \ref{StabDet} we need to show that for every $P\in\B_s^+(X,X^*)$ the element $L(P)=\Phi(P)-\Phi'(P)P=PBB^*P+C^*C$ is an $L^1$ detector for the Lyapunov semigroup generated by $\L_A+\Phi'(P)=\L_{A-BB^*P}$, see \eqref{ShiftLyap} and Corollary \ref{LyapAdd}. Setting $K:=B^*P\in\B(X,U)$ and noting that $K^*K=PBB^*P$ is symmetric we see that this is exactly the conclusion of Lemma \ref{ExpDetL2}.
\end{proof}
Setting $K=0$ in Lemma \ref{ExpDetL2} one can also see that exponential detectability implies detectability in $L^2$. Specializing Theorem \ref{Lyappi} to our current setting we prove Theorem \ref{Wonh} from Section \ref{S4} next. It is a generalization of the stability criterion in terms of Lyapunov equations proved for reflexive spaces in \cite{K2}, the matrix case is due to Wonham \cite[12.4]{Wh}. In fact, we prove a stronger claim that only asumes that $(C,A)$ is detectable in $L^2$.
\begin{proof}[Proof of {\rm\bf Theorem \ref{Wonh}}] By inspection and Lemma \ref{StabSub}, $\X=\B_s(X,X^*)$, $\T(t)$, $z=C^*C$ and $\F=(X\ohat X)^+$ satisfy the conditions of Theorem \ref{Lyappi}. We see that Theorem \ref{Lyappi}(i) is equivalent to the (i) here. Claims (ii),(iii) of Theorems \ref{Lyappi} are equivalent to (ii),(iii) here as well since the exponential stability of $\T(t)$ is equivalent to that of $T(t)$.
\end{proof}

We now turn to stabilizability. The map $\Phi(P)=-PNP+Q$ is clearly Gateaux differentiable with $\Phi'(P)R=-PNR-RNP$, and the auxiliary maps from \eqref{F'LPsi} become
\begin{align}\label{LPsi}
L(P)&=\Phi(P)-\Phi'(P)P=PNP+Q\\ \Psi(P,Q)&:=\Phi(P)-\Phi(R)+\Phi'(P)(R-P)=(P-R)N(P-R).\notag
\end{align}
The last formula implies in particular that $\Phi$ is order concave when $N\geq0$ since $\Psi(P,Q)\geq0$ is equivalent to the definition of order concavity for Gateaux differentiable maps. For the full Riccati map $F(P)=\L_AP+\Phi(P)$ the formal derivative is equal to 
$$
F'(P)R=\L_AR+\Phi'(P)R=A^*R+RA-PNR-RNP=(A^*-PN)R+R(A-NP).
$$
Although $P^*\neq P$ in non-reflexive spaces we still have $(B^*P)^*=PB$ because the domain of $B$ is a Hilbert space, so $B^{**}=B$, and $\Ran(B)\subset X$ while $P^*\big|_{X}=P$. Therefore also $(NP)^*=(BB^*P)^*=(B^*P)^*B^*=PBB^*=PN$, and using Corollary \ref{PNR} we conclude that 
\begin{equation}\label{ShiftLyap}
F'(P)R=(A-NP)^*R+R(A-NP)=\L_{A-NP}R\,.
\end{equation}
This means that $(\L_A,\Phi)$ is exponentially stabilizable in the sense of Definition \ref{StabDet} if and only if there is $P\in\B_s^+(X,X^*)$ such that $A-NP$ is exponentially stable. For $N=BB^*$ we will reduce this to existence of $K\in\B(X,X^*)$ such that $A-BK$ is exponentially stable, which is the standard definition of a pair $(A,B)$ being exponentially stabilizable \cite[2.4]{Wh}, \cite[I.2.5]{Zb}. 
\begin{lemma}\label{ExpStabIff} Let $A$ be the generator of a positive $C_0$ semigroup on $X$, $B\in\B(U,X)$ and $C\in\B(X,V)$ for some Hilbert spaces $U,V$. Suppose that $(C,A)$ is exponentially detectable and set $\Phi(P):=-PBB^*P+C^*C$. Then $(\L_A,\Phi)$ is exponentially stabilizable in $\B_s^+(X,X^*)$ if and only if $(A,B)$ is exponentially stabilizable.
\end{lemma}
\begin{proof} In this proof the lower indices under both semigroup symbols $T(t)$ and $\T(t)$ indicate the corresponding generator of $T(t)$. Since $\L_A+\Phi'(P)=\L_{A-BB^*P}$ one direction is trivial. If $\T_{A-BB^*P}(t)$ is exponentially stable then so is $T_{A-BB^*P}(t)$ and one can set $K:=B^*P$.

Conversely, let $T_{A-BK}(t)$, and hence $\T_{A-BK}(t)$, be stable for some $K\in\B(X,U)$. Consider the Lyapunov equation
$\L_{A-BK}P=-(C^*C+K^*K)$. The cone $(X\ohat_\pi X)^+$ is generating and by Lemma \ref{ExpDetL2} $C^*C+K^*K$ is an $L^1$ detector for $\T_{A-BK}(t)$, so by Theorem \ref{Lyappi} there exists a solution $P\in\B_s^+(X,X^*)$. By Corollary \ref{PNR} we have 
\begin{multline}
\L_{A-BB^*P}P=\L_{A-BK}P+\L_{BK-BB^*P}P\\
=-(C^*C+K^*K)+(BK-BB^*P)^*P+P(BK-BB^*P)\\
=-C^*C-K^*K+K^*B^*P-PBB^*P+PBK-PBB^*P\\
=-\Big(C^*C+PBB^*P+(K-B^*P)^*(K-B^*P)\Big).
\end{multline}
This is another Lyapunov equation that $P$ solves. By Corollary \ref{GlobL1Det}, $C^*C+PBB^*P$ is an $L^1$ detector for $\L_{A-BB^*P}$, all the more so after adding $(K-B^*P)^*(K-B^*P)\geq0$. Existence of a solution $P\geq0$ guarantees by Theorem \ref{Lyappi} that $\T_{A-BB^*P}(t)$ is exponentially stable, so this $P$ provides the required stabilization of $(\L_A,\Phi)$.
\end{proof} 
We are ready to prove the main theorem.
\begin{proof}[Proof of {\rm\bf Theorem \ref{WonhBan}}] Let $\L_A$ be the Lyapunov generator corresponding to $A$, i.e. $\L_AP=A^*P+PA$ on $\D_A$, and let $\Phi(P):=-PBB^*P+C^*C$. By Corollary \ref{GlobL1Det} and Lemma \ref{ExpStabIff} the pair $(\L_A, \Phi)$ is exponentially stabilizable and globally $L^1$ detectable. The choice of $P_0$ is justified by Lemma \ref{ExpStabIff}. Obviously, $\Phi(0)\geq0$, and $P\mapsto\Phi(P)$ and $(P,R)\mapsto\Phi'(P)R=-PBB^*R-RBB^*P$ are continuous on bounded monotone sequences by Corollary \ref{PNR}. Therefore, by Theorem \ref{NewtRicc} there is a unique positive definite solution $P$ to the Riccati equation, which is stabilizing and a weak* monotone limit of solutions $P_n$ to \eqref{LyapIter}. Since $P_n$ is monotone decreasing and bounded by Theorem \ref{Ls+}(iii) it also strong operator converges to the same limit. It remains to prove the quadratic convergence estimate.

Let $F(P):=A^*P+PA-PBB^*P+C^*C$ on $\D_A$. By Lemma \ref{Iter}(ii) we have $F(P_{n+1})=-\Psi(P_n,P_{n+1})$, where $\Psi$ is as in \eqref{F'LPsi} with $N=BB^*$. Therefore, by concavity and \eqref{F'LPsi}
$$
F'(P)(P_{n+1}-P)\geq F(P_{n+1})-F(P)=-\Psi(P_n,P_{n+1}).
$$
Since $F'(P)=\L_{A-BB^*P}$ and $P$ is stabilizing the semigroup generated by $\L_{A-BB^*P}$ is exponentially stable, and the inverse $-F'(P)^{-1}$ is bounded and positive by Theorem \ref{Lyappi}(iii). Therefore, 
$0\leq P_{n+1}-P\leq-F'(P)^{-1}\Psi(P_n,P_{n+1})$. Since the norm on $\B_s(X,X^*)$ is monotone by Theorem \ref{Ls+}(i) we estimate
$$
\|P_{n+1}-P\|\leq\|F'(P)^{-1}\|\|\Psi(P_n,P_{n+1})\|.
$$
It follows from \eqref{F'LPsi}, the inequalities $0\leq P_n-P_{n+1}\leq P_n-P$, and the monotonicity of the norm that
$$
\|\Psi(P_n,P_{n+1})\|=\|(P_n-P_{n+1})BB^*(P_n-P_{n+1})\|\leq\|B\|^2\|P_n-P_{n+1}\|^2\leq\|B\|^2\|P_n-P\|^2.
$$
Setting $\kappa:=\|(\L_{A-BB^*P})^{-1}\|\|B\|^2$ we obtain the estimate.
\end{proof}

\end{document}